\documentclass[a4paper,11pt]{book}
\usepackage[T1]{fontenc}
\usepackage{amsmath,amsthm,amssymb}
\usepackage[utf8]{inputenc}
\usepackage[frenchb]{babel}
\usepackage[all]{xy}

\def \R{{\mathbb R}}

\def \1{{\mathbb 1}}

\newtheorem{Prop}{Proposition}
\newtheorem{Def}{Définition}
\newtheorem{Rem}{Remarque}
\newtheorem{Lem}{Lemme}

\newtheorem{Def Nota}{Definitions and notations} 
\newtheorem{Cor}{Corollaire}

\font\ninerm=cmr9
\long\outer\def\abstract#1{\bigskip\vbox{\noindent\ninerm
\baselineskip=10pt#1}\nobreak\bigskip}
\newcount\numero
\def\exo#1{\advance\numero by 1\bigskip
{\noindent\tenbf #1\the\numero. }}
\def\frac#1#2{{#1\over #2}}
\title{Une réflexion à partir de la Nature de Spinoza : "La substance ou la Nature comme Treillis".}
\author{Par Mohammed Bachir en hommage \`a Baruch de Spinoza.} 
\date{11/03/2018}

\begin{document}
\maketitle
\tableofcontents
\chapter{Existence et unicité de la substance}
\vskip10mm
\hspace{4cm}{\it "Ni rire, ni pleurer, ni haïr mais comprendre."

\hspace{9cm} Spinoza, {\it Traité politique.}}\\
\vskip10mm
\noindent {\bf Abstract :} We propose a simple mathematical model based on two axioms and the set theory to approach the problem developed by the philosopher B. Spinoza in {\it the Ethics}. We then use the Knaster-Tarski Theorem to prove the existence and uniqueness of the Substance asserted by Spinoza.

\vskip5mm

\noindent {\bf Résumé :} On propose un mod\`ele math\'ematique simple bas\'e sur deux axiomes et la théorie des ensembles pour  approcher la probl\'ematique d\'evelopp\'ee par le philosophe B. Spinoza, dans {\it l'Ethique}. On utilise ensuite le Th\'eor\`eme de Knaster-Tarski pour d\'emontrer l'existence et l'unicit\'e de la Substance affirm\'ees par Spinoza.

\vskip5mm 

{\it Je tiens à remercier le Professeur Daniel Parrochia pour avoir lu et commenté une première version de cette note. Ces remarques et suggestions m'ont été d'une aide précieuse.} 

\section{Introduction}

Vers 1661, dans l'\oe uvre de {\it l'Ethique} (Voir \cite{KFM}), Baruch de Spinoza commence par d\'efinir la notion de ``{\it cause de soi}'' et celle de ``{\it substance}''. Par ''{\it cause de soi}'' il entend, ce dont l'essence enveloppe l'existence, ou encore ce dont la Nature ne peut \^etre con\c cue qu'existante. Par substance, ce qui est en soi et est con\c cu par soi, autrement dit, ce dont le concept n'a pas besoin du concept d'une autre chose pour \^etre form\'e. Il affirme ensuite qu'il appartient \`a la nature de la substance d'exister (Voir les d\'efinitions I et II ainsi que  la Proposition VII dans {\it l'Ethique} \cite{KFM}). Le philosophe met en place sept axiomes sur lesquels il va s'appuyer tout le long de la premi\`ere partie de l'\oe uvre pour d\'emontrer (selon ce qu'il appelle lui m\^eme la m\'ethode g\'eom\'etrique) l'existence et l'unicit\'e d'une substance dans la Nature (Voir Proposition XI  de {\it l'Ethique}). Autrement dit, de l'existence d'une unique chose dans la Nature qui est cause de soi. Cependant, Spinoza  est amen\'e a aborder la notion de substance via ses attributs. Le seul moyen pour lui de percevoir la substance par l'entendement. Quoique le raisonnement qu'on trouve dans {\it l'Ethique} soit souvent rigoureux et l'intuition profonde, certaines preuves pr\'esentent n\'eanmoins quelques lacunes \`a certains endroits comme le d\'emontre George Boole dans son \oe uvre ``les lois de la pens\'ee'' (voir \cite{Bo}, voir aussi ''{\it La raison systématique}'' de Daniel Parrochia \cite{P}. Je prouverai dans la deuxième partie, que les démonstrations de la Proposition XI de l'Ethique données par Spinoza sont en réalité fausses). Ces difficult\'es s'expliquent \`a mon sens pour les raisons suivantes :

\vskip5mm

$1.$ La première difficulté est due au fait qu'on ne peut raisonner sur la substance que par l'intermédiaire de ses attributs, ce qui peut induire des erreurs et des confusions.

\vskip5mm

$2.$ la seconde difficult\'e est due \`a l'absence d'une th\'eorie math\'ematique des ensembles au XVIIe siècle. M\^eme si on trouve d\'ej\`a dans l'argumentation de Spinoza une distinction entre deux types d'infini : ``{\it l'infini en son genre}'' et ``{\it l'absolument infini}'', la th\'eorie des ensembles n'était pas encore constituée. Notons au passage que c'est \`a Georg Cantor que l'on doit, deux siècles plus tard vers 1874, les premières bases de la th\'eorie des ensembles. Voici un passage de {\it l'Ethique} (Voir page 324 dans \cite{KFM}) o\`u Spinoza critique ses adversaires. On y voit clairement que la th\'eorie des ensembles et la notion d'ordinal n'\'etaient pas encore au point mais que certains problèmes qu'on retrouvera plus tard constituaient déjà implicitement le fond des débats : \\

 ``{\it Si la substance corporelle, disent-ils, est infinie, qu'on la con\c coive divis\'ee en deux parties : chaque partie sera ou finie ou infinie. Si chaque partie est finie, l'infini se compose donc de deux parties finies, ce qui est absurde. Si chaque partie est infinie, il y a donc un infini deux fois plus grand qu'un autre infini, ce qui est \'egalement absurde. En outre, si une quantit\'e infinie est mesur\'ee en parties valant chacune un pied, elle devra consister en une infinit\'e de telles parties; et de m\^eme, si elle est mesur\'ee en parties valant chacune un pouce; par suite, un nombre infini sera douze fois plus grand qu'un autre nombre infini...}'' \\\\

La pr\'esente note trouve sa source d'inspiration dans l'\oe uvre de {\it l'Ethique}.  Ayant constat\'e que le raisonnement de Spinoza s'appuie en partie sur des concepts ensemblistes, je propose dans ce qui suit un mod\`ele math\'ematique simple issu de la th\'eorie des ensembles et bas\'e sur deux axiomes concernant la loi de la causalit\'e (en dehors des axiomes classiques de la th\'eorie des ensembles) qui tente de développer une approche formelle de la problématique de Spinoza (Voir la Section \ref{S2} et la Section \ref{S4}). Cependant, contrairement \`a Spinoza, j'appelle "{\it Nature}" $N_t$ \`a un instant $t$, uniquement l'univers de toutes les  choses $A$ qui existent ou ayant exist\'e avant l'instant $t$ et qui ne co\"{i}ncident pas avec leur propre cause (un principe fondamental de la physique). Je d\'efinis ensuite formellement $E_t$ comme \'etant l'ensemble de toutes les choses existantes ou ont existé avant l'instant $t$, englobant ainsi \'eventuellement les choses qui co\"{i}ncideraient avec leurs propres causes si toutefois ces choses l\`a avaient été amenées à exister. Par ailleurs, une fois que les objets trait\'es seront bien d\'efinis, la mod\'elisation et les axiomes seront bien pos\'es, les preuves seront purement math\'ematiques.
 
\vskip5mm
\'Evidemment, il est constant que $N_t$ est une partie de $E_t$, pour tout instant $t$. La premi\`ere question qui se pose alors est la suivante : Est-il vrai que $E_t=N_t$ ou bien $N_t$ est strictement inclus dans $E_t$ ? Je d\'emontrerai en m'appuyant sur ces deux axiomes, qu'il existe n\'ecessairement une unique chose non vide $S$ dans $E_t$ qui existe en soi. Autrement dit, il existe une unique chose qui est cause de soi, c'est-\`a-dire dont l'existence et la cause ne font qu'un. Il appartient donc \`a la nature de cette chose d'exister : C'est la Substance de Spinoza. Je prouverai par la m\^eme occasion que $S$ est ind\'ependante du temps et que pour tout instant $t$ on a en fait $E_t=N_t \cup S $. Notons, que dans les Propositions XII et XIII de {\it l'Ethique}, Spinoza affirme que la substance est indivisible. Il sera aussi prouv\'e dans cette note que $S$ est effectivement indivisible, c'est-\`a-dire qu'elle ne se transforme pas en deux (ou plusieurs) parties indépendantes (Voir Proposition \ref{Substance} en Section \ref{S4}).\\

Il est \`a noter que la Substance $S$ établie dans cette note, appara\^{i}t comme ext\'erieur \`a la Nature $N_t$ \`a chaque instant $t$, alors que chez Spinoza la notion de Nature doit plut\^ot co\"{i}ncider avec $E_t$ dans sa globalit\'e \`a chaque instant $t$. Ces deux approches semblent, en apparence, diverger mais en réalité elle différent surtout dans la formulation. Ceci étant dit, le but de cette note n'est pas nécessairement de formaliser à l'identique la pensée de Spinoza, mais de proposer un modèle mathématique qui approche le mieux possible la pensée de l'auteur de {\it l'Ethique}. En effet, ce que j'appelle la substance ici, doit correspondre à la "Nature naturante" chez Spinoza. Quant à la "Nature naturée", elle n'est rien d'autre chez Spinoza que la modification de la substance par la nécessité de sa nature. A partir de là, si on considère $E_t$ comme une structure globale, rien n'empêche qu'elle coïncide avec la "Nature naturée" de Spinoza. Il est juste à préciser, que $E_t$ est considérée dans cette note uniquement d'un point de vu ensembliste sans structure globale. D'où la différence entre $E_t$ à un instant $t$ et un état primordial qui correspond à ce que j'ai appelé la substance $S$. Autrement dit, la cause de tout ce qui a existé avant hier n'est pas égal à tout ce qui a existé avant hier mais probablement inclus dans tout ce qui a existé avant l'an dernier. Alors qu'on montrera que la cause de $S$ est égal à $S$. En d'autre terme, la "Nature naturante" est antérieur à la "Nature naturée". Il est \`a noter aussi que, contrairement \`a Spinoza qui cherche \`a connaitre la substance par son essence et ses attributs, on ne pr\'etend pas dans cette note expliquer de quoi est faite la nature de cette substance. Il est uniquement question de prouver formellement et gr\^ace \`a des outils math\'ematiques qu'elle existe et qu'elle est unique. L'id\'ee est quelque peu similaire \`a l'exemple suivant : dans un langage fini (un dictionnaire par exemple) o\`u l'on admet que tout mot possède une d\'efinition unique, il est alors n\'ecessaire qu'il y ait au moins un mot qui est sa propre d\'efinition s'il l'on veut éviter de d\'efinir un premier mot par un deuxième et puis ce même deuxième par le premier, autrement dit si l'on veut \'eviter que la chaine des d\'efinitions fasse une boucle. En terme de causalité, il est admis de Spinoza et des physiciens classiques, que le temps n'est pas cyclique. En d'autres termes, le fait que l'effet ne pr\'ec\`ede jamais sa cause est un principe fondamental sur lequel repose toute la physique. 
 
\vskip5mm

La preuve que j'apporterai reposera sur le th\'eor\`eme de Knaster-Tarski qui est un r\'esultat math\'ematique de la th\'eorie des points fixes d\^u aux math\'ematiciens, Knaster \cite{K} (1928) et Tarski \cite{T} (1955). Ce r\'esultat est \'egalement  attribu\'e par Y. N. Moschovakis \cite{M} \`a Zermelo. L'inter\^et de l'approche d\'evelopp\'ee dans cette note, c'est qu'elle permet de revisiter {\it l'Ethique} de Spinoza via les math\'ematiques modernes que l'\'epoque de Spinoza ignorait, tout en confirmant les th\`eses soutenues par ce dernier dans les Propositions XI, XII, XIII et XIV de {\it l'Ethique}.

\vskip5mm

Voici les deux premiers axiomes que je propose.

\vskip5mm

$(A1)$ Toute chose existante est l'effet d'une cause bien d\'et\'ermin\'ee (toujours sous entendue non vide) et l'effet ne peut pas pr\'ec\'eder sa cause. (Voir aussi les axiomes de {\it l'Ethique}).

\vskip5mm

\noindent{\it Explication : C'est un principe fondamental de la physique.}

\vskip5mm

$(A2)$ La cause du tout doit envelopper la cause de la partie. (Cet axiome ne se trouve pas explicitement dans {\it l'Ethique}).

\vskip5mm

\noindent{\it Explication : Cet axiome dit que si une chose enveloppe une autre alors la cause de la premi\`ere enveloppe la cause de la seconde.  Lorsqu'une chose est consid\'er\'ee comme compos\'ee de plusieurs parties, sans chacune des parties le tout ne peut pas exister. Donc la cause d'une partie est contenue dans la cause du tout. Par exemple, la cause d'une page d'un livre fait partie de la cause du livre étant donné que sans ses pages, le livre ne peut exister. }

\vskip5mm

Gr\^ace aux deux axiomes $(A1)$ et $(A2)$ et une modélisation mathématique simple, on prouvera dans la Proposition \ref{Substance1} : 

\paragraph{L'existence de la substance :}
{\it Il existe une substance $S$ dans le monde de l'existant. Tout autre substance (si elle existe), contiendrait la substance $S$. Autrement dit, $S$ est la plus petite substance contenue dans tout autre substance s'il y a plusieurs substances.}

\vskip5mm
Je rappelle à ce niveau que Spinoza \'enonce directement \`a partir de la d\'efinition d'une substance, la Proposition II de {\it l'Ethique} : {\it "Deux substances ayant des attributs diff\'erents n'ont rien en commun entre elles"}. Si on admet cette proposition comme \'etant intuitivement vraie, on en tire directement l'unicit\'e de la substance. Car d'apr\`es ce que j'ai \'enonc\'e plus haut (ce qui sera prouv\'e ult\'erieurement), toute autre substance si jamais elle existait, contiendrait en elle la substance $S$ qui existe n\'ecessairement. Or si deux substances diff\'erentes ne doivent rien avoir en commun, sachant qu'elles ont toutes $S$ en commun c'est qu'en fait, il n y a pas d'autre substance que $S$.  Ainsi on aura fini avec l'existence et l'unicit\'e de la substance $S$ \`a partir uniquement des deux axiomes $(A1)$ et $(A2)$. Mais comme je n'ai pas défini une substance à partir de ses attributs, comme le fait Spinoza, mais uniquement comme étant une chose qui est cause de soi, je suggère à la place de la Proposition II de de {\it l'Ethique}, l'axiome équivalent suivant :

\vskip5mm

$(A3)$ Deux substances différentes n'ont rien en commun.

\vskip5mm

\noindent{\it Explication : Comme le précise Spinoza dans sa démonstration de la Proposition II de {\it l'Ethique}, chacune, en effet, doit exister en soi et doit être conçue par soi, autrement dit le concept de l'une n'enveloppe pas le concept de l'autre.}

\vskip5mm
Si on assume donc l'axiome $(A3)$ en plus des axiomes $(A1)$ et $(A2)$, voici alors ce qu'on d\'emontrera (voir Proposition \ref{Substance}) : 

\paragraph{Unicit\'e et indivisibilit\'e de la substance :}
{\it Il existe une unique substance indivisible $S$ dans le monde de l'existant.} 
\vskip5mm
Notons qu'a partir de ce qu'on vient d'énoncer, on déduit le caractère infini de la substance et ceci découle immédiatement de la définition donnée par Spinoza, c'est-à-dire que la substance $S$ n'est pas contenue dans une autre substance de même nature. C'est clairement le cas puisqu'il n'y a pas d'autre substance que $S$. Le fait que $S$ enveloppe la {\it cause première} sera démontré dans le Corollaire \ref{Cor1}, une fois qu'on s'accorde à définir la {\it cause première} comme étant la cause contenue dans toutes les autres causes.
\vskip5mm

La démarche qui sera suivit dans les prochaines sections se résume ainsi : En admettant les trois axiomes $(A1)$, $(A2)$ et $(A3)$ ainsi qu'une modélisation mathématique, on partira du multiple c'est-à-dire de la diversité des êtres de l'existence, pour arriver via les mathématiques à l'Un et l'unique c'est-à-dire à la substance. 
\section{ Une mod\'elisation math\'ematique.} \label{S2} 

Commen\c cons par mod\'eliser le temps par la droite r\'eelle not\'ee $\R$.

\subsection{Quelques notations et d\'efinitions.}
Pour chaque instant $t\in \R$ arbitrairement fix\'e : 



$\bullet$ on appelle cause d'une chose $A$, la chose qui a fait ou qui fait que $A$ existe ou a existé.

$\bullet$ on dit qu'une chose a une cause en dehors d'elle m\^eme, si cette chose ne co\"{i}ncide pas avec sa cause.

$\bullet$ on note $N_t$ la Nature avant l'instant $t$, autrement dit l'ensemble de tous ce qui peut \^etre consid\'er\'e comme un ''\^etre'' ou un ''objet'' ayant exist\'e avant l'instant $t$ et dont la cause le précède strictement dans le temps, en particulier qui ne co\"{i}ncide pas avec sa propre cause.

$\bullet$ on note $E_t$ l'ensemble de tous ce qui peut \^etre consid\'er\'e comme un ''\^etre'' ou un ''objet'' existant avant l'instant $t$. On a toujours $N_t\subset E_t$ pour tout $t\in \R$. Par ''chose'', on d\'esigne toute partie de $E_t$ c'est-\`a-dire un sous-ensemble de $E_t$ qui peut toutefois \^etre un singleton.

$\bullet$ par $(\mathcal{P}(E_t),\subset)$ on d\'esigne l'ensemble de toutes les parties de $E_t$ muni de l'ordre de l'inclusion.

$\bullet$ on note $\mathcal{C}(A)$ la cause bien d\'et\'ermin\'ee de $A$, pour tout $A\in \mathcal{P}(E_t)$ et pour tout $t \in \R$. L'ensemble $\mathcal{C}(A)$ peut d\'ependre de l'instant $t$, en revanche la loi de causalit\'e $\mathcal{C}$ est suppos\'ee \^etre immuable, autrement dit est ind\'ependante du temps. La cause d'un \'el\'ement $x\in E_t$ est tout simplement la cause de $\lbrace x \rbrace \in \mathcal{P}(E_t)$.

$\bullet$ on appelle substance toute partie $S\subset E_t$ ($t\in \R$) qui co\"{i}ncideraient avec sa propre cause, c'est-à-dire $\mathcal{C}(S)=S$ si toutefois une telle chose existe.

$\bullet$ on note par $E$, le monde de l'existant c'est-\`a-dire $E= \cup_{t\in \R} E_t$.

$\bullet$ on note $e$ la chose commune \`a tout ce qui existe. Autrement dit, la chose contenue dans toute autre chose de $E_t$ pour tout $t\in \R$ c'est-\`a-dire $$e:= \cap_{t\in \R}\cap_{A\subset E_t} A.$$
Cet ensemble $e$ est appel\'e en math\'ematiques l'ensemble vide et est not\'e $\emptyset$ ou $\lbrace\rbrace$. Cependant, nous gardons pour l'instant la notation $e$ au lieu de $\emptyset$, mais nous pr\'ecisons que $e$ a exactement les m\^emes propriétés  mathématiques que $\emptyset$. Notons juste que l'ensemble $e$ c'est-à-dire l'ensemble vide, peut ici avoir deux interprétations équivalentes. La première désigne négativement "{\it l'absence de choses commune}" à ce qui existe dans $E_t$, d'où la notation de $e$ par $\emptyset$. Mais la deuxième interprétation désigne positivement "{\it la diversité des êtres}" dans $E_t$. Dire que la chose commune $e$ dans $E_t$ est vide, autrement dit est inexistante, revient à dire que les êtres dans $E_t$ sont divers.  L'ensemble $e$ aurait donc pu tout aussi bien être noté "$d$" comme diversité. 

\vskip5mm
On rappelle que, pour tout $t\in \R$, $(\mathcal{P}(E_t),\subset)$ est un treillis complet (et m\^eme une alg\`ebre de Boole compl\`ete)  qui admet un plus petit \'el\'ement $e=\emptyset$ et un plus grand \'el\'ement $E_t$. L'introduction de l'ensemble $\mathcal{P}(E_t)$ se justifie par les passages suivant que l'on trouvera dans {\it l'Ethique} \cite{KFM} pages 315-316 :

\vskip5mm

{\it 1. Que la vraie d\'efinition de chaque chose n'enveloppe et n'exprime rien \`a part la nature de la chose d\'efinie. D'o\`u il suit :

2. Que nulle d\'efinition n'enveloppe et n'exprime aucun nombre d\'etermin\'e d'individus, puisqu'elle n'exprime rien d'autre que la nature de la chose d\'efinie. Par exemple, la d\'efinition du triangle n'exprime rien d'autre que la simple nature du triangle, mais non quelque nombre d\'etermin\'e de triangles.

3. Il faut noter que, de chaque chose existante, il est n\'ecessairement donn\'e quelque cause d\'etermin\'ee, par laquelle elle existe.

4. Il faut enfin noter que cette cause par laquelle certaine chose existe doit, ou bien \^etre contenue dans la nature m\^eme et la d\'efinition de la chose existante (parce que, en effet, il appartient \`a sa nature d'exister), ou bien \^etre donn\'ee en dehors d'elle.

Cela pos\'e, il suit que, s'il existe dans la Nature quelque nombre d\'etermin\'e d'individus, une cause doit n\'ecessairement \^etre donn\'ee pourquoi ces individus existent, et pourquoi ni un plus grand nombre ni un plus petit.}

\vskip5mm

Ce dernier passage exprime le fait que les causes respectives de deux individus $\lbrace a\rbrace$ et $\lbrace b\rbrace$, n'enveloppent pas la cause de la paire d'individus $\lbrace a,b \rbrace$. Autrement dit, il est possible que $\mathcal{C}(\lbrace a\rbrace) \cup \mathcal{C}(\lbrace b \rbrace) \subset \mathcal{C}(\lbrace a,b \rbrace)$ mais que $\mathcal{C}(\lbrace a,b \rbrace)\neq \mathcal{C}(\lbrace a\rbrace) \cup \mathcal{C}(\lbrace b \rbrace)$. Il fallait donc prendre en consid\'eration l'ensemble des parties des \^etres ou objets existants, ce qui a motiv\'e l'introduction de $\mathcal{P}(E_t)$. Il faut aussi noter que si deux choses n'ont rien en commun entre elles, c'est qu'il y a une cause non vide et bien déterminée \`a cela. Ceci explique entre autre pourquoi la cause du vide ne peut pas \^etre vide.

\vskip5mm
\subsection{Traduction mathématique des deux premiers axiomes.}
Commen\c cons par traduire math\'ematiquement les axiomes $(A1)$ et $(A2)$.

\vskip5mm
$(A1)$ Toute chose existante est l'effet d'une cause bien déterminée (toujours sous-entendue non vide) et l'effet ne peut pas pr\'ec\'eder sa cause. Autrement dit, pour tout $t\in \R$ et pour tout $A\in \mathcal{P}(E_t)$, il existe une cause unique $\mathcal{C}(A)\in \mathcal{P}(E_t)\setminus\lbrace \emptyset\rbrace$. Cela signifie que la loi de causalité $\mathcal{C} : \mathcal{P}(E_t) \longrightarrow \mathcal{P}(E_t)\setminus\lbrace \emptyset\rbrace$ est une application, pour tout $t\in \R$.

\vskip5mm

$(A2)$ La cause du tout doit envelopper la cause de la partie. Cela signifie que pour deux choses existantes, si l'une enveloppe l'autre alors la cause de la premi\`ere enveloppe la cause de la seconde. Autrement dit : pour tout $t\in \R$ et pour tout $A, B \in \mathcal{P}(E_t)$, si $A\subset B$ alors $\mathcal{C}(A)\subset \mathcal{C}(B)$. Cela signifie que la loi de causalité $\mathcal{C} : \mathcal{P}(E_t) \longrightarrow \mathcal{P}(E_t)$ est croissante, pour tout $t\in \R$.

\vskip5mm

Remarquons qu'il est tout à fait possible d'unifier les deux axiomes $(A1)$ et $(A2)$ en un seul axiome de la manière suivante : 
\vskip5mm

{\bf Un axiome unifié :} Il existe une loi de causalité $\mathcal{C}$ telle que : $\forall t \in \R,$ la loi $ \hspace{2mm} \mathcal{C} : \mathcal{P}(E_t) \longrightarrow \mathcal{P}(E_t)\setminus \lbrace  \emptyset \rbrace$ est une application croissante pour l'ordre de l'inclusion.

\vskip5mm

C'est en fait cet axiome unifié qui est la clé de l'existence d'une substance non vide dans le monde de l'existant et ceci grâce au célèbre théorème de Knaster-Tarski comme nous le verrons plus loin.
\subsection{Deux premières propositions fondamentales.}
\paragraph{A) Le sens de l'ensemble $e$ :} Rappelons que l'ensemble contenu dans tout sous-ensembles non vide $A$ de $E_t$ pour tout $t \in \R$ s'écrit
$$\emptyset=e:= \cap_{t\in \R}\cap_{A\subset E_t} A,$$
et consiste en l'intersection de tous les sous ensembles non vides $A$ de $E_t$ ($t\in \R$), ce qui aboutit sur l'ensemble vide. Comme les êtres de $E_t$ sont différents les uns des autres, dire par exemple que $e=\lbrace x \rbrace \cap \lbrace y \rbrace$ équivaut tout simplement à dire que $x\neq y$, pour tout $x, y \in E_t$ et pour tout $t\in \R$. Ainsi, évoquer l'ensemble vide $e$ revient à faire référence à la diversité des êtres dans $E_t$. En effet, $e$ n'a de sens ici que parce qu'il existe des êtres différents dans $E_t$. On peut donc dire que $e$ autrement dit l'ensemble vide, exprime la diversité des êtres dans $E_t$ pour tout $t\in \R$. De ce fait, parler de la cause de $e$ revient à parler de la cause de la diversité des êtres :  $\mathcal{C}(e):=$ " La cause de la diversité des êtres de $E_t$ pour $t\in \R$". Ceci semble en total accord avec la pensée de Spinoza, puisque ce dernier stipule que les essences des êtres sont éternelles (mais pas leur existences évidement), de ce fait, la diversité des êtres existant se manifesterait déjà dans leur éternelles essences. En conséquence, on a que $\mathcal{C}(e)\neq e$ vue qu'il doit exister par l'axiome $(A1)$ une cause bien déterminée (sous entendu non vide) à cette diversité des êtres, puisque cette diversité est bien une réalité existante dans $E_t$ pour chaque instant $t\in \R$. Autrement dit, il doit exister une cause bien déterminée et non vide qui explique pourquoi il existe des choses différentes dans le monde de l'existant et non pas une unique chose. On obtient ainsi la première proposition.

\begin{Prop} \label{IV} \label{symp1} On a que $\mathcal{C}(e)\neq e$ ou, ce qui revient au même $\mathcal{C}(\emptyset)\neq \emptyset$. Autrement dit, la cause de la diversité des êtres de $E_t$ pour $t\in \R$, est non vide.

\end{Prop}

L'explication du sens de cette proposition a été livré plus haut et sa démonstration repose comme je l'ai mentioné sur l'axiome $(A1)$. On peut cependant en donner une autre démonstration de la manière suivante : 
\begin{proof}[D\'emonstration 2. :]  De par la d\'efinition m\^eme de l'ensemble vide c'est qu'il est d\'epourvu de toute chose et donc d\'epourvu de toute cause. Ce qui est d\'epourvu de toute cause ne peut \^etre une cause (m\^eme pas de lui m\^eme).
\end{proof}

On peut aussi fournir une troisième démonstration, qui cette fois repose sur l'axiome VII de {\it l'Ethique} qui stipule que : {\it "tout ce qui peut \^etre con\c cu comme non existant, son essence n'enveloppe pas l'existence"}. Autrement dit, tout ce qui peut \^etre con\c cu comme non existant, ne peut \^ etre cause de soi. Or ici $e=\emptyset$ exprime le fait que ce qui est commun \`a toutes les choses de $E_t$ ($t\in \R$) est inexistant. Il s'en suit donc que $e$ ne peut être sa propre cause. D'o\`u  $\mathcal{C}(e)\neq e$ par cet axiome VII de Spinoza. Notons encore une dernière fois, que l'ensemble vide $e=\emptyset$ prend sens dès qu'il existe au moins deux êtres différents dans $E_t$ pour un certain instant $t\in \R$. Par conséquent, dire que la chose commune à tout ce qui existe est inexistante revient à dire que les choses existantes sont diverses dans leur ensemble et ceci doit avoir une cause bien déterminée et non vide qui l'explique. 

\paragraph{B) La cause de la diversité est contenue dans toutes les causes.} Sans trop m'attarder sur l'aporie de "l'Un et le multiple", je reprends ici la formule de Hegel sur l'identité :
\begin{center}
{\it "Aussi l'absolu lui-même est-il l'identité de l'identité et de la non-identité."}
\end{center}
Autrement dit, ce qui fait qu'une chose est, c'est à la fois ce qu'elle est et ce qu'elle n'est pas. Cette formule trouve toute sa légitimité en terme de causalité en lien avec l'axiome $(A2)$ proposé dans cette note. Partant de la formule de Hegel adaptée à la causalité, on peut expliquer philosophiquement ce qui se déduira mathématiquement de l'axiome $(A2)$, c'est-à-dire, pourquoi la cause de la diversité, autrement dit $\mathcal{C}(e)$, doit être contenue dans toutes les causes. En effet, pour qu'un être $x$ soit différent d'un autre être $y$, il est nécessaire que la cause $\mathcal{C}(\lbrace x \rbrace)$ qui précède $x$, contienne en elle la raison qui explique pourquoi $x$ existera et pourquoi il sera différent de $y$. En d'autre terme, $\mathcal{C}(x \neq y)=\mathcal{C}(e=\lbrace x \rbrace \cap \lbrace y \rbrace)=\mathcal{C}(e)$ doit être contenu dans $\mathcal{C}(\lbrace x \rbrace)$ et aussi dans $\mathcal{C}(\lbrace y \rbrace)$. Donc, la cause de $x$ c'est-à-dire $\mathcal{C}(\lbrace x \rbrace)$, doit contenir en elle non seulement ce qui fera $x$ mais aussi l'explication de ce que $x$ ne sera pas. C'est uniquement ainsi qu'on peut expliquer pourquoi $x$ existe et pourquoi il est différent des autres êtres. Ce raisonnement s'étend bien entendu à toute partie $A\neq \emptyset$ de $E_t$ ($t\in \R$) qui n'est pas nécessairement un singleton. Ainsi, on obtient la proposition suivante dont la preuve mathématique repose tout simplement sur l'axiome $(A2)$ :

\begin{Prop} \label{V} Supposons que l'axiome $(A2)$ est satisfait. Soit $t\in \R$ et $A \in \mathcal{P}(E_t)$. Alors, $\mathcal{C}(e) \subset \mathcal{C}(A)$. Autrement dit, la cause de la diversité des êtres (c'est-\`a-dire la cause de $e=\emptyset$) est contenue dans toutes les causes.

\end{Prop}

\begin{proof} Partons de l'axiome $(A2)$, comme $e\subset A$ pour tout $\emptyset \neq A\subset E_t$ ($t\in E_t)$ et que la loi de causalité est croissante, alors  $\mathcal{C}(e) \subset \mathcal{C}(A)$. 

\end{proof}

\section{Le th\'eor\`eme de Knaster-Tarski.} \label{S3}

Le th\'eor\`eme de Knaster-Tarski, qu'on utilisera plus loin, est un r\'esultat bien connu de la th\'eorie des ensembles. Il concerne l'existence de points fixes pour une application croissante d'un treillis complet dans lui m\^eme. On donnera plus bas son \'enonc\'e ainsi qu'une d\'emonstration. On commence par rappeler ci-dessous quelques notions bien connues concernant les treillis.\\

\noindent {\bf Treillis :} Un treillis $(T,\leq)$ est la donn\'ee d'un ensemble $T$ et d'une relation not\'ee $\leq$ (appel\'ee ordre partiel) satisfaisant les propri\'et\'es suivantes :

$(1)$ $\forall x \in T$, $x\leq x$.

$(2)$ $\forall x, y \in T$, si $x\leq y$ et $y\leq x$ alors $x=y$. 

$(3)$ $\forall x, y, z  \in T$, si $x\leq y$ et $y\leq z$ alors $x\leq z$.

$(4)$ pour tous éléments $x$ et $y$  de $T$, il existe une borne supérieure et une borne inférieure à l'ensemble $\lbrace x, y \rbrace$.

\vskip5mm

\noindent {\bf Treillis complet :} Un treillis complet $(T,\leq)$ est un treillis pour lequel toute partie $A$ de $T$ admet une borne sup\'erieure not\'ee $\sup(A)\in T$ et une borne inf\'erieure not\'ee $\inf(A)\in T$. Autrement dit:

$(i)$ $\forall a \in A$, $a\leq \sup(A)$ (respectivement $\inf(A)\leq a$),

$(ii)$ s'il existe $M\in T$ tel que $\forall a \in A$, $a\leq M$ alors $\sup(A)\leq M$ (respectivement s'il existe $m\in T$ tel que $\forall a \in A$, $m \leq a$ alors $m\leq \inf(A)$)

\vskip5mm

\noindent{\bf Exemple :} Soit $E$ un ensemble non vide et $\mathcal{P}(E)$ l'ensemble des parties de $E$. Alors $(\mathcal{P}(E),\subset)$ est un treillis complet, o\`u $\subset$ d\'esigne l'inclusion. 

\vskip5mm

On rappelle ci-dessous le th\'eor\`eme de Knaster-Tarski.

\begin{Lem}\label{KT} {\bf (Knaster-Tarski)} Soit $(B,\leq)$ un Treillis complet et $\mathcal{C}: B \rightarrow B$ une application croissante pour l'ordre $\leq$.  Alors, il existe respectivement un plus petit et un plus grand point fixe $\omega, \theta \in B$, autrement dit : $$ \mathcal{C}(\omega)=\omega, \hspace{2mm} \mathcal{C}(\theta)=\theta$$ et pour tout élément $\lambda\in B$ tel que $\mathcal{C}(\lambda)=\lambda$ on a $$\omega\leq \lambda\leq \theta.$$ 
 
\end{Lem}

\begin{proof}[Preuve :] $\bullet$ {\it \bf Le plus petit point fixe : } Posons $I:=\left\{a\in B : \mathcal{C}(a)\leq a\right\}$. Comme $B$ est un Treillis, il contient l'\'el\'ement $\omega:= \inf(I)$.  Comme $\omega\leq a$ pour tout $a\in I$, la croissance de $\mathcal{C}$ entraîne que $\mathcal{C}(\omega)\subset \mathcal{C}(a)$ et comme $a\in I$ alors $\mathcal{C}(\omega)\leq a$ et ceci pour tout $a\in I$, il s'en suit que $\mathcal{C}(\omega)\leq \inf(I) =\omega$. On conclut que $\omega\in I$. Posons maintenant $\varpi:=\mathcal{C}(\omega)$. On a donc $\varpi\leq \omega$. La croissance de $\mathcal{C}$ entraîne que $\mathcal{C}(\varpi)\leq \mathcal{C}(\omega):=\varpi$. D'o\`u $\varpi\in I$. Puisque $\omega$ est le plus petit \'el\'ement de $I$ alors $\omega\leq \varpi$. Comme on avait d\'ej\`a \'etabli l'in\'egalit\'e inverse alors $\varpi=\omega$ c'est-\`a-dire $\mathcal{C}(\omega)=\omega$. D'o\`u l'existence d'un point fixe. Soit maintenant un \'el\'ement $\lambda\in B$ tel que $\mathcal{C}(\lambda)=\lambda$. On a en particulier que 
$\lambda\in I$ et donc $\omega\leq \lambda$. L'\'el\'ement $\omega$ est donc le plus petit point fixe de $\mathcal{C}$.\\
$\bullet$ {\bf Le plus grand point fixe :} Soit  maintenant $J:=\left\{a\in B : a \leq \mathcal{C}(a)\right\}$. Comme $B$ est un Treillis, il contient l'\'el\'ement $\theta:= \sup(J)$. Alors $a\leq \theta$ pour tout $a\in J$ et par la croissance de $\mathcal{C}$ on a $\mathcal{C}(a)\leq \mathcal{C}(\theta)$ et puisque $a\in J$ alors $a \leq \mathcal{C}(\theta)$ pour tout $a\in J$. D'o\`u par passage \`a la réunion, $\theta\leq \mathcal{C}(\theta)$. Ainsi $\theta\in J$. En posant $\delta:=\mathcal{C}(\theta)$ et en utilisant la croissance de $\mathcal{C}$, on montre comme plus haut que $\theta=\mathcal{C}(\theta)$. L'\'el\'ement $\theta$ est le plus grand point fixe. En effet, soit $\lambda$ tel que $\mathcal{C}(\lambda)=\lambda$. En particulier $\lambda\in J$ et donc $\lambda\leq \theta$.\\
$\bullet$ {\it \bf Conclusion : } On a au final d\'emontr\'e  que $\omega=\mathcal{C}(\omega)\leq \theta=\mathcal{C}(\theta)$ et tout autre point fixe $\lambda$ v\'erifie $\omega\leq \lambda\leq \theta$. \\
\end{proof}

\begin{Rem} Le th\'eor\`eme de Knaster-Tarski a une application au th\'eor\`eme de Cantor-Bernstein (voir \cite{wiki}) qui est bien connue.
\end{Rem}

\section{Existence et unicit\'e de la substance de Spinoza.}\label{S4}

On \'enonce maintenant les r\'esultats principaux de cette note, les propositions qui confirment les th\`eses propos\'ees par Spinoza dans les Propositions XI, XII, XIII et XIV de {\it l'Ethique}. 
Par la proposition qui suit, on montre qu'il existe n\'ecessairement une chose qui est cause de soi. 

\begin{Prop} \label{Substance1} Supposons que les axiomes $(A1)$ et $(A2)$ sont satisfaits. Alors, il existe une substance $S\subset \cap_{t\in \R}E_t$, distincte de $e$ i.e $S\neq e$ et contenant la cause de la diversité :

$$\mathcal{C}(e)\subset \mathcal{C}(S)=S,$$ autrement dit $S$ est une substance ou encore, qu'il appartient \`a la nature de $S$ d'exister. Tout autre substance $\Lambda\subset E_t$ pour tout $t\in \R$, contiendrait $S$, c'est-à-dire, $S\subset \Lambda$.

\end{Prop}

\begin{proof}[Preuve :]  Gr\^ace aux axiomes $(A1)$ et $(A2)$, on sait que la loi de causalité $\mathcal{C}: \mathcal{P}(E_t)\longrightarrow \mathcal{P}(E_t)$ est bien une application croissante, pour tout $t\in \R$. Comme $\mathcal{P}(E_t)$ est un Treillis complet, il existe par le Lemme \ref{KT} un plus petit point fixe $S_t\in \mathcal{P}(E_t)$ i.e. telle que $\mathcal{C}(S_t)=S_t$ et un plus grand point fixe $D_t=\mathcal{C}(D_t)$. Tout autre point fixe $\Lambda\in \mathcal{P}(E_t)$ satisfait $ S_t\subset \Lambda \subset D_t$.

{\it Ind\'ependance du temps de $S_t$ :} En effet pour $t\leq t'$, on sait que $E_t\subset E_{t'}$ et donc aussi que $\mathcal{P}(E_t)\subset \mathcal{P}(E_{t'})$. Comme on vient de le montrer, il existe un plus petit et un plus grand point fixe $S_t, D_t\in \mathcal{P}(E_t)$
ainsi qu'un plus petit et un plus grand point fixe $S_{t'}, D_{t'}\in \mathcal{P}(E_{t'})$. D'une part, comme $\mathcal{P}(E_t)\subset \mathcal{P}(E_{t'})$, on se trouve donc avec les points fixes $S_t, D_t, S_{t'}, D_{t'}\in \mathcal{P}(E_{t'})$. Ainsi, puisque $S_{t'}, D_{t'}\in \mathcal{P}(E_{t'})$ sont respectivement le plus petit et le plus grand point fixe on obtient que 
$$ S_{t'} \subset S_t \subset D_t\subset D_{t'}.$$

Il s'en suit que $S_{t'}$ existe aussi avant $t$ i.e. $S_{t'}\in \mathcal{P}(E_{t})$, puisque $S_{t'}$ est contenue dans  $S_t\in \mathcal{P}(E_{t})$. Or dans $\mathcal{P}(E_t)$, le plus petit  point fixe \'etant $S_t$, tout autre point fixe de $\mathcal{P}(E_t)$ envelopperait $S_t$. Autrement dit on a aussi 
$$ S_{t} \subset S_{t'}.$$

Donc, $S_{t} = S_{t'}$. Ceci signifie que $S=S_t$ ($\forall t\in \R$) est ind\'ependante du temps. D'où, $S\subset \cap_{t\in \R}E_t$. Comme $e\subset S$ alors par la Proposition \ref{V} (ou par l'axiome $(A2)$) on obtient $\mathcal{C}(e)\subset \mathcal{C}(S)=S$ et par la Proposition \ref{IV} on a que $S\neq e$. 


\end{proof}

Comme je l'ai d\'ej\`a mention\'e dans l'introduction, on peut d\'eduire directement l'unicit\'e de la substance $S$ \`a partir de la Proposition \ref{Substance1} si on admet (comme le fait Spinoza dans la Proposition II de {\it l'Ethique}) que deux substances diff\'erentes ne doivent rien avoir en commun (voir l'axiome $(A3)$), car effectivement chacune doit \^etre cause de soi et exister par soi. D'où la proposition suivante. 

\begin{Prop} \label{Substance2} Supposons que les axiomes $(A1)$, $(A2)$ et $(A3)$ sont satisfaits. Alors, il existe une unique substance $S$ indépendante du temps et contenant la cause de la diversité. Autrement dit, il existe une unique substance $S$ qui satisfait $\mathcal{C}(e)\subset S\subset \cap_{t\in \R} E_t$.
\end{Prop}

\begin{proof} L'existence d'une substance $S$ telle que $\mathcal{C}(e)\subset S\subset \cap_{t\in \R} E_t$, est assur\'ee par la Proposition \ref{Substance1}. Soit $t\in \R$ quelconque et $\Lambda\subset E_t$ une substance. Il s'agit de voir que $\Lambda=S$. En effet, supposons par l'absurde que $\Lambda\neq S$. Par la même Proposition \ref{Substance1}, on sait que $S\subset \Lambda$. Ainsi, les deux substances $S$ et $\Lambda$ ont toutes les deux $S$ en commun. Or, ceci est exclu par l'axiome $(A3)$, ce qui donne $\Lambda=S$. D'o\`u l'unicité de la substance $S$.

\end{proof}

\vskip5mm
On montre maintenant l'indivisibilit\'e de la substance. 

\begin{Prop} \label{Substance} Supposons que les axiomes $(A1)$, $(A2)$ et $(A3)$ sont satisfaits. Alors, l'unique substance $S$ satisfait :

$(1)$ pour tout $t\in \R$, $E_t=N_t\cup S $. Autrement dit, les êtres de l'existence sont ou bien en soi, et c'est l'unique substance S, ou bien en autre chose que soi et, dans ce cas, ce sont les \'el\'ements de $N_t$.  

$(2)$ $S$ est indivisible, c'est-\`a-dire, $S$ ne se d\'ecompose pas en deux (ou plusieurs) parties disjointes. 
\end{Prop} 

\begin{proof} L'existence et l'unicit\'e de la Substance $S$ ainsi que son ind\'ependance par rapport au temps, est assur\'ee par la Proposition \ref{Substance2}.

$(1)$ pour tout $t\in \R$, $S$ est l'unique chose qui est cause de soi. Il est donc clair que pour tout $t\in \R$, $E_t=N_t\cup S $. 

$(2)$ Supposons que $S$ se divise en $S_1$ et $S_2$. On va procéder ici en suivant la démarche de la Proposition XIII de {\it l'Ethique}. Deux cas se présentent : 

\vskip5mm
{\it Cas 1 :}  Soit $S_1$, soit $S_2$ garde la même nature que $S$. Autrement dit, soit $S_1$ est une substance, soit $S_2$ est une substance. Si, par exemple, $S_1$ est une substance, alors par la Proposition \ref{Substance1}, elle devrait contenir $S$, autrement dit $S\subset S_1$. Or, $S_1\subset S$ et donc $S_1=S$. Le même raisonnement s'applique si $S_2$ est une substance.

{\it Cas 2 :} Ni $S_1$ ni $S_2$ n'est une substance. Dans ce cas, comme le précise Spinoza dans la Proposition XIII de {\it l'Ethique}, la substance $S$ perdrait sa nature de substance et cesserait d'être ce qui contredit le fait d'être cause de soit, autrement dit, contredit le fait qu'il appartiennent à sa nature d'exister.

\end{proof}

\vskip5mm
On appellera "Cause premi\`ere", la cause contenue dans toutes les causes.

\begin{Cor} \label{Cor1} La substance $S$ enveloppe la cause première de toute autre chose et elle est éternelle.
\end{Cor}
\begin{proof} Soit une chose $A\subset E_t$ pour $t\in \R$. Par la Proposition \ref{V}, on a $\mathcal{C}(e)\subset \mathcal{C}(A)$. Donc, $\mathcal{C}(e)$ est la cause première qui qui est contenue dans la substance $S$. Par ailleurs, $S$ est cause de soi et est ind\'ependante du temps elle est donc \'eternelle.
\end{proof}
En suivant la définition II de {\it l'Ethique} ci-dessous, on déduit aisément l'infinité de la substance $S$.  

\begin{Def} Une chose est dite finie en son genre quand elle peut être bornée par une autre chose de même nature. Par exemple, un corps est dit chose finie, parce que nous concevons toujours un corps plus grand ; de même, une pensée est bornée par une autre pensée ; mais le corps n'est pas borné par la pensée, ni la pensée par le corps. 
\end{Def}
\begin{Cor} \label{Cor2} L'unique substance $S$ est nécessairement infinie. 
\end{Cor}
\begin{proof} Cela découle de la définition ci-dessus puisqu'il n'y a pas d'autre sustance de même nature que $S$ pour la contenir.
\end{proof}

\section{La substance est-elle un atome?}
Dans cette section, on se pose la question suivante : La substance $S$, qu'on sait indépendante du temps et contenue dans  $\cap_{t\in \R} E_t$, est elle un singleton ou non? 

\vskip5mm
On va discuter une condition possible qui permet de répondre à cette questions. On rappel la définition d'un atome en théorie des ensembles.
\begin{Def} \label{atome} En théorie des ensembles, un ensemble $A$ est dit un {\it atome}, si $A\neq \emptyset$ et $A$ n'a pas d'autre minorant que $\emptyset$ et lui
même (c'est-à-dire que les seuls ensembles $B$ vérifiant $B\subset A$ sont $\emptyset$ et $A$).
\end{Def}

Notons tout d'abord deux choses : 
\vskip5mm
$(1)$ l'unique substance $S$ constitue un infra-monde où le temps n'existe pas. Cependant, les axiomes $(A1)$ et $(A2)$ restent toujours  valide. La chose à noter c'est que l'axiome $(A1)$ exclue uniquement le fait qu'un effet puisse précéder dans le temps sa cause mais n'exclue pas, par exemple, qu'un effet puisse exister en même temps que sa cause. Ainsi, si donc $S$ contient une partie non vide $A\subsetneqq S$, on peut toujours parler de sa cause $\mathcal{C}(A)\neq A$ qui est contenue dans $S=\mathcal{C}(S)$ et cela indépendamment du temps. 
\vskip5mm
$(2)$ Les trois axiomes $(A1)$, $(A2)$ et $(A3)$ ne permettent pas de démontrer mathématiquement parlant que la substance $S$ est un atome ou non. Pour pouvoir affirmer que $S$ soit un atome, il faut supposer une condition en plus qu'on notera $(P)$. On n'affirme pas ici que cette propriété $(P)$ est vrai, car l'intuition ne semble ni l'affirmer ni l'infirmer, mais nous allons en revanche montrer que la validité la propriété $(P)$ entraine le fait que $S$ soit un atome et plus précisément que $S=\mathcal{C}(e)$.
\vskip5mm
Voici donc la propriété $(P)$ en question :
\vskip5mm

\noindent $(P)$ Pour tout $t \in \R$ et pour tout $A\subset E_t$, on a  $$A\subset \mathcal{C}(A)\Longrightarrow (A=e \textnormal{ ou } A=\mathcal{C}(A)).$$

\vskip5mm
Ce qui est certain c'est que, pour les choses singulières $A$ de la nature $N_t$ ($t\in \R$), c'est-à-dire, pour les choses qui n'ont pas toujours existé et qui ont besoin que leur causes les précèdent strictement dans le temps, on ne peut jamais avoir $A\subset \mathcal{C}(A)$ sauf si $A=\emptyset$, car par l'axiome $(A1)$ une chose ne peut jamais précédé sa cause. Or si $A\subset \mathcal{C}(A)$, cela suppose que $A$ fait partie de sa propre cause. Ceci est impossible pour les choses singulières qui n'ont pas toujours existé. 

\begin{Prop} \label{Substance4} Supposons que les axiomes $(A1)$, $(A2)$ et $(A3)$ sont satisfaits. Alors, on a que $(i)\Longrightarrow (ii)$.

$(i)$ La propriété $(P)$ est vérifiée.

$(ii)$ L'unique substance $S$ est un atome (au sens de la définition \ref{atome}), c'est-à-dire qu'il existe $s\in \cap_{t\in \R} E_t$ tel que $S=\lbrace s\rbrace$ et on a que $S=\mathcal{C}(e)=\lbrace s\rbrace$.

\end{Prop}

\begin{proof} Montrons que $(i)\Longrightarrow (ii)$. On sait que $e\subset \mathcal{C}(e)$. Par l'axiome $(A2)$, c'est-à-dire par la croissance de la loi de causalité, on sais que  $\mathcal{C}(e)\subset \mathcal{C}(\mathcal{C}(e))$. Si donc la propriété $(P)$ est vrai, alors on obtient que ou bien $\mathcal{C}(e)=e$ ou bien que $\mathcal{C}(e)$ est une substance. Comme la Proposition \ref{symp1} exclu le premier cas, on a alors que $S=\mathcal{C}(e)$ (car $S$ est l'unique substance par Proposition \ref{Substance2}). Pour voir que $S$ est un atome, soit $A\subset S$ alors on sait que $\mathcal{C}(e)\subset \mathcal{C}(A)\subset \mathcal{C}(S)=S$. D'où $\mathcal{C}(A)=S$. Il s'ensuit que $A\subset S=\mathcal{C}(A)$. Ceci entraine par la propriété $(P)$ que $A=e$ ou $A=S$, c'est-à-dire que $S$ est un atome.
\end{proof}

\section{Conclusion.}

Comme je l'ai mentionné à la fin de l'introduction, la démarche qui a été suivie dans cette note consistait en le fait de partir de la diversité des êtres pour aboutir, via un modèle mathématique, à l'unicité de la substance. Il a été prouvé par ailleurs que la substance est indivisible, infinie, éternelle et qu'elle est cause première de toute chose. Il serait donc bien intéressant maintenant de faire le chemin inverse, c'est-à-dire, de partir de la substance elle-même dont on a prouvé l'existence et l'unicité pour comprendre dans quelle mesure le modèle proposé ici permet d'expliquer le cheminement causal qui part de cette substance et aboutit à la diversité des êtres. Ceci permettrait sans doute de voir si la suite des affirmations de Spinoza dans {\it l'Ethique} reste valide. Mais cette t\^{a}che n'a pas été abordée dans cette note pour éviter toute spéculation non mathématique. Notons juste que le formalisme mathématique qui a été utilisé ici est possible dans un cadre plus général. En effet, dès que l'on suppose qu'il existe un ordre $\leq$ sur le monde de l'existant $E$ faisant de ce dernier un treillis complet et qu'il existe une loi de causalité de $\mathcal{C}: (E, \leq)\longrightarrow (E, \leq)$ qui soit croissante pour cet ordre, alors grâce au théorème de Knaster-Tarski, il existe nécessairement une substance dans $E$ et même que l'ensemble des substances est lui même un treillis complet. L'unicité de la substance s'obtient ensuite en rajoutant l'axiome suivant :  {\it "deux substances différentes ne sont pas comparables"}. Cela laisse donc la voie libre à d'autre modèle possible. Ceci étant dit, j'ai choisi de garder ici le modèle simple qu'est $(\mathcal{P}(E_t), \subset)$ car ce dernier ma semblé être le plus adéquat et le plus concret à l'heure actuelle.
\vskip5mm  

Je vais m'arrêter sur cette conclusion que les premières affirmations de Spinoza trouvées dans {\it l'Ethique} (je parle uniquement de ce que j'ai abordé ici) ont été confirmées et démontrées dans cette note par une approche rigoureusement mathématique. On peut donc probablement encore dire comme l'a affirmé Antonio Damasion dans l'un de ses livres (voir \cite{D}) que : Spinoza avait raison.

\chapter[Les erreurs de Spinoza]{Les erreurs de Spinoza dans ses démonstrations de la Proposition XI de l’Ethique I.}
\noindent {\bf Abstract :} We prove that the three demonstrations of Proposition XI of the {\it Ethics I} proposed by Spinoza are false.

\vskip5mm

\noindent {\bf Résumé :} On prouve que les trois démonstrations de la Proposition XI de l'{\it Ethique I} proposées par Spinoza sont fausses.

\section{Introduction.}
Dans la précédent partie j’ai proposé un modèle mathématique simple, basé sur la théorie des ensembles et le théorème de Knaster-Tarski pour tenter d’approcher la pensée développée par Spinoza dans l’Ethique I. J’ai apporté ensuite des preuves rigoureusement mathématiques  à quelques propositions de l’Ethique I notamment à deux des plus importantes qui sont la Propositions XI et la Proposition XIV. J’avais mentionné dans la même partie que certaines preuves données par Spinoza étaient erronées. D'autres auteurs comme Georges Boole dans \cite{Bo} et Daniel Parrochia dans \cite{P} avaient par ailleurs déjà mentionné quelques exemples de difficultés que l’on peut trouver dans l’Ethique. Dans la présente note, je vais commenter quelques lacunes qu'on trouve dans l'Ethique et prouver ensuite que les trois démonstrations de la Propositions XI de l’Ethique I  proposées par Spinoza sont fausses.

\vskip5mm 

Avant de commencer les détails de mon argumentaire,  je tiens à préciser que la présente note n’a nullement la prétention de mettre en défaut la pensée de l’auteur de l’Ethique. Je soutiens, au contraire, que l’œuvre de l’Ethique soit une mine  d’idées et de concepts. Mais comme chacun le sait, toute mine d’or à sa proportion en terre et en or  et ce n’est pas parce qu’on tombe plus facilement sur de la terre qu’on ne finit  pas par  trouver de l’or. Je pense avoir suffisamment prouvé dans la première partie que même si les démonstrations de Spinoza contiennent des erreurs, il n’en demeure pas moins que son intuition reste très profonde. La première leçon que tout lecteur de Spinoza  a dû tirer de l’Ethique c’est justement  de ne jamais suivre une doctrine avant de  l’avoir soigneusement inspecté et de l’avoir rigoureusement soumis à l’examen de la preuve. C’est ce que Spinoza lui-même a fait vis-à-vis de la philosophie de Descartes. Enfin, il est important de noter que l’œuvre de l’Ethique date de 1661, une époque à laquelle il n’existait ni la logique de Georges Boole (1854), ni la théorie des ensembles de George Cantor (1874) ni les théorèmes d'incomplétude de Kurt Gödel (1931). 

\section{Le point de départ.}\label{S0}
Pour Spinoza, le point de départ c’est que dans la nature il n y a que des substances et leur modes. C’est un fait admis de manière axiomatique et donc non démontrable, basé sur l’axiome 1 et  les définitions 3 et 5 de l’Ethique. En effet, dans sa démonstration de la Proposition IV, Spinoza dit clairement : « Toutes les choses qui sont, sont ou bien en soi, ou bien en autre chose (selon l’axiome 1), c’est-à-dire (selon les définitions 3 et 5) que rien n’est donné hors de l’entendement, à part les substances et leurs affections». On se demande alors qu’est-ce que Spinoza essaye de démontrer dans la Proposition XI  vu qu’il suppose axiomatiquement l’existence de substances dans la nature. Le souci pour Spinoza dans la Proposition XI est de prouver qu’il en existe une qui admet une infinité d’attributs. Il tentera ensuite de prouver dans la proposition XIV que cette substance  est unique. Cependant, on verra  dans ce qui suit que les démonstrations  de la Proposition XI fournis  par Spinoza sont en réalité fausses. 
\subsection{\`A la nature de la substance, il appartient d’exister.}

Afin de prouver plus loin que les démonstrations de la Proposition XI proposées par Spinoza sont erronées, je vais revenir sur les propositions antérieurs à celle-ci et leur donner leur sens exacts.  On verra ainsi comment Spinoza lui-même tombe dans le piège de ces propres mots. Revenons donc sur la Proposition VII que Spinoza utilise dans sa première démonstration de la Proposition XI.
\vskip5mm   
PROPOSITION VII : {\it À la nature de la substance, il appartient d’exister.}
\vskip5mm
Comme chacun le sait, Spinoza ne veut évidemment  pas dire par cette proposition  que dans la nature il existe nécessairement une  substance. Car comme je l’ai précisé plus haut, l’existence de substance(s) est un fait axiomatique pour Spinoza. Le vrai sens de la Proposition VII c’est que : lorsqu’une substance est donné dans la nature, il est inutile de chercher d’où provient-elle. Car elle aura tout simplement toujours existé parce qu’elle est cause de soi et donc que son essence enveloppe son existence c’est-à-dire qu’il appartient tout simplement à sa nature d’exister. Autrement dit, qu’elle n’a ni surgit à partir du {\it "Rien"} ni a  été produite par une autre substance. Mais ce qu’il faut bien noter ici et ce à quoi il faut faire attention c’est qu’il est question dans  la Proposition VII de vrai substance. Autrement dit, si une substance selon la définition III de l'Ethique existe réellement dans la nature, alors il appartient à sa nature d’exister. Car, ce n’est pas parce qu’on définit une chose qu’elle va nécessairement exister dans la nature. Il s’en suit donc, que la Proposition VII s’applique uniquement aux substances qui existent réellement dans la nature et non pas aux substances imaginé par l’esprit.  Pour éviter le piège dans lequel Spinoza lui-même va tomber  (les détails seront donnés plus loin), je vais reformuler la Proposition VII pour lui faire dire ce qu’elle veut vraiment  dire : 
\vskip5mm
PROPOSITION VII BIS: {\it Si une substance est donnée dans la nature (et non imaginée par l’esprit), alors à la nature de cette substance, il appartient d’exister.}
\vskip5mm
Les termes « si » et « alors » sont ici très important comme nous le verrons dans ce qui suit.
Notons que ce qu’on vient de faire remarquer concerne toutes les propositions de l’Ethique qui font référence aux substances. A chaque fois il est question de substance supposée existante dans la nature et non pas imaginée inexistante comme il est question dans la démonstration 1. de la Proposition XI de l'Ethique I (voir plus loin). Ainsi, par exemple la Proposition II de l’Ethique : « Deux substances ayant des attributs différents n’ont rien de commun entre elles » doit être comprise comme suit :
\vskip5mm
Proposition II Bis : « Si deux substances existent dans la nature et ont des attributs différents, alors  elle n’ont rien de commun entre elles» 
\vskip5mm
La aussi il est clair que les termes « si » et « alors » sont incontournables puisque Spinoza sait pertinemment qu'à la fin, il y aura qu'une seule substance.
\section{Les erreurs de Spinoza.}
Voici la Proposition XI de l’Ethique et la première démonstration donnée par Spinoza.
\vskip5mm

PROPOSITION XI : Dieu, autrement dit une substance consistant en une infinité d’attributs, dont chacun exprime une essence éternelle et infinie, existe nécessairement. 
\vskip5mm
AXIOME VII. – TOUT CE QUI PEUT ÊTRE CONÇU COMME NON EXISTANT, SON ESSENCE N’ENVELOPPE PAS L’EXISTENCE.
\vskip5mm
DÉMONSTRATION 1 : Si vous niez cela, concevez, s’il est possible, que Dieu n’existe pas. Donc (selon l’axiome 7) son essence n’enveloppe pas l’existence. Or (selon la proposition 7) cela est absurde : donc Dieu existe nécessairement. C.Q.F.D.
\subsection{Première preuve de l’erreur dans la démonstration 1.}\label{S1}
L’erreur commise par Spinoza est contenue dans la dernière partie de sa preuve. En effet, comme je l’ai fais savoir plus haut, la Proposition VII ne peut s’appliquer qu’aux substances supposées existantes réellement dans la nature et non aux substances imaginées non existantes. Or Spinoza commence sa démonstration en supposons que la substance Dieu (selon sa définition) est non existante et dans la mesure où c’est une substance non existante elle est supposé être une substance imaginée (Autrement dit, une substance par le vocabulaire et non pour de vrai). On ne peut donc en aucun cas lui appliquer la Proposition VII qui traite uniquement les substances supposées existantes. En fait, la démonstration 1 de Spinoza se résume à ceci : « Supposons que Dieux n’existe pas. Or cela est absurde : donc Dieu existe nécessairement ». Ce qui n'est évidemment pas une démonstration. Même en supposons par l'absurde que Dieu n'existe pas, autrement dit que la substance Dieu est imaginaire, Spinoza utilise quand même la Proposition VII, or il est évident qu'il ne peut pas appartenir à la nature d'une substance non existante, d'exister. 
\subsection{Deuxième preuve de l’erreur dans la démonstration 1.}\label{S3.2}
Pour voir d’une autre manière que la démonstration 1 donnée par Spinoza est erronée, je vais imaginer une autre substance que j’appellerais Dieu Bis et que je définie comme suis :
\vskip5mm
PAR DIEU BIS, J’ENTENDS UNE SUBSTANCE CONSISTANT EN EXACTEMENT {\bf DEUX D’ATTRIBUTS NI PLUS NI MOINS}, DONT CHACUN EXPRIME UNE ESSENCE ÉTERNELLE ET INFINIE.  
\vskip5mm
Le lecteur pourra facilement se convaincre, qu’en remplaçant "Dieu" par "Dieu Bis" dans la démonstration 1, il tombera sur la même conclusion c’est-à-dire que Dieu Bis existe nécessairement. On peut ensuite répéter ce processus avec une substance à exactement un attribut, à trois attributs…, à $n$ attributs etc et à chaque fois une telle substance existera nécessairement. 
Ce qui exclu de fait toute unicité de la substance et contredit la Proposition XIV de l'ethique. 
\vskip5mm
Hormis l'invalidité de la démonstration 1 pour les raisons données juste avant et aussi dans la Section \ref{S1}, la notion d'attribut pose de sérieux problèmes notamment des problèmes de rigueurs. En effet, il est impossible dans le système de Spinoza de décider sur la cardinalité des attributs des substances. En d'autres termes, pour qu’on puisse espérer avoir des informations sur le nombre ou la cardinalité des attributs d'une substance (encore faut-il que les attributs existent réellement), il est nécessaire que ces informations proviennent de quelques part et en l'occurrence ici, de ce qui précède la Proposition XI de l'Ethique I. Or, les seules endroits où Spinoza pense expliquer la provenance du nombre des attributs d'une substance se trouve dans la définition IV qui est utilisée ensuite dans la Proposition IX. Mais, d'une part la définition IV ne dit rien sur le nombre des attributs d'une substance et d'autre part, Spinoza lui-même explique par ailleurs que l'entendement ne perçoit que deux attributs de la substance (l'étendu et la chose pensante). Je ne vois donc pas comment l'entendement pourrait {\bf percevoir qu'il y ait plus d'attributs que ce qu'il peut percevoir}. De cela on déduit clairement, que si la démonstration 1 de Spinoza était juste, rien ne pourrait réfuter l'existence de Dieu Bis, au contraire l'entendement serait plus satisfait avec uniquement deux attributs puisqu'il n'en perçoit pas d'autres.

Rien donc ne peut, dans le système de Spinoza, décider de manière rigoureuse sur la cardinalité des attributs d’une substance (sans évoquer ici le fait qu’il y ait en plus plusieurs type de cardinalité). L'infinité des attributs de la substance (si cela devait avoir un sens) doit à mon sens être considéré comme un axiome dans le système de Spinoza et ne peut ainsi se déduire d’une démonstration. Or Spinoza, comme je l'avais expliqué, admet l'existence de substance(s) de manière axiomatique. Il s'ensuit donc que la Proposition XI de l'Ethique est en fait un axiome dans le système Spinosiste. Quant à l'approche que j'ai proposée dans la première partie, j'ai clairement prouvé de manière mathématique l'existence et l'unicité de la substance (que j'ai défini uniquement comme étant une chose qui est cause de soi), et cela à partir uniquement de deux axiomes (voir un troisième pour l'unicité). 
\subsection{Deux erreurs dans la deuxième démonstration de Spinoza.}
On trouve dans la deuxième démonstration de Spinoza pour la Proposition XI deux types d'erreurs que je vais mettre à jour. 
La première erreur se trouve dans le passage suivant de sa démonstration : 
\vskip5mm
« Si donc nulle raison ni cause ne peut être donnée qui empêche que Dieu n'existe, ou qui lui enlève l'existence, il faut conclure pleinement qu'il existe nécessairement. »
\vskip5mm
On sait depuis Kurt Gödel, que dans une théorie donné, il existe toujours des propositions qui ne sont ni démontrable ni réfutable. L'axiome du choix est  un des plus célèbre  exemple qui ne peut ni être  démontré  ni être réfuté  à partir  de  la théorie  des ensembles  de  Zermelo-Fraenkel . Le passage ci-dessus (de Spinoza) est donc un faux argument car ce n'est pas parce qu'on ne peut pas prouver  "l'absence de Dieu" dans l'axiomatique de Spinoza qu'il existe nécessairement. Je pense par ailleurs que la Proposition XI telle qu'elle est formulée dans l’Ethique I est indécidable dans le système axiomatique de Spinoza.
\vskip5mm
La deuxième erreur est de même nature que ce que j'ai développé dans la section \ref{S1}. Reprenons le passage suivant de la deuxième démonstration de Spinoza : 
\vskip5mm
« D'autre part, une substance qui serait d'une autre nature ne pourrait avoir rien de commun avec Dieu (selon la proposition 2), et par conséquent ne pourrait ni poser son existence ni l'enlever. »
\vskip5mm
Là aussi, Spinoza utilise la Proposition II de l’Ethique pour dire qu’une autre substance que la substance Dieu n’aurait rien en commun avec Dieu. Seulement, comme je l’avais expliqué dans la la section \ref{S0}, la Proposition II de l’Ethique  traite les substances qui sont  supposées toutes les deux existantes. Or ici, le Dieu de Spinoza n’est pas supposé existant (c'est ce qu'il souhaite plutôt montrer) il est bien au contraire supposé, par l’absurde, être non existant. D’où l’erreur.
\subsection{L’erreur dans la troisième démonstration de Spinoza.}
L’argument de Spinoza dans  sa troisième démonstration commence comme suit : 
\vskip5mm
« Ne pouvoir exister, c’est impuissance, et au contraire pouvoir exister, c’est puissance (comme il est connu de soi). Si donc ce qui existe déjà nécessairement, ce ne sont que des êtres finis, c’est donc que des êtres finis sont plus puissants que l’Être absolument infini : or cela (comme il est connu de soi) est absurde...».
\vskip5mm
Si on suit cet argument de Spinoza, en définissant des demi-dieux intermédiaires en puissance entre les humains et le Dieu de Spinoza ou bien en considérant le "Dieu Bis" que j'ai défini dans la Section \ref{S3.2}, on montrera par-là que  "Dieu Bis" existe nécessairement puisqu'il est plus puissant que les humains qui existent déjà nécessairement. D'où la faille. Je ne vais pas détailler davantage cette partie, car cette troisième preuve de Spinoza repose sur l’{\it argument ontologique} qui a déjà été réfuté à juste titre par plusieurs penseurs dont les philosophes Emmanuel Kant et Bertrand Russell.

\bibliographystyle{amsplain}

\end{document}